\documentclass[11pt,a4paper]{amsart}
\usepackage{graphics}
\usepackage{color}
\usepackage{epic}
\usepackage[latin1]{inputenc}
\newtheorem{theo}{\bf Theorem}[section]
\newtheorem{propo}[theo]{\bf Proposition}
\newtheorem{lemma}[theo]{\bf Lemma}

\newtheorem{coro}[theo]{\bf Corollary}

\theoremstyle{remark}

\theoremstyle{definition}
\newtheorem{definition}[theo]{\bf Definition}

\begin{document}
\title {\bf Conjugacy of Coxeter elements}
\author{{\sc Henrik Eriksson} and {\sc Kimmo Eriksson}}
\thispagestyle{empty}   

\begin{abstract}
For a Coxeter group $(W,S)$, a permutation of the set $S$ 
is called a Coxeter word and the group element represented by the product
is called a Coxeter element. Moving the first letter to the end 
of the word is called a rotation and two Coxeter elements are rotation 
equivalent if their words can be transformed into each other through a
sequence of rotations and legal commutations.

We prove that Coxeter elements are conjugate if and only if they are
rotation equivalent. This was known for some special cases but not for
Coxeter groups in general.
\end{abstract}
\maketitle
\section{Introduction}
\noindent
Consider the Coxeter group defined by the Coxeter graph below. A 
{\em Coxeter word}
is a list of all generators in any order, so there are 24 Coxeter words in 
our example. Interpreting words as products we get 12 different 
Coxeter elements ($s_0$ commutes with $s_2$ and $s_3$), which fall into two
different conjugacy classes.
 
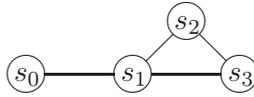
\begin{figure}[h]
 \begin{picture}(80,10)(-20,0)
    \put(-40,0){\circle{14}}
    \put(-44.5,-2.5){$s_0$}
    \put(-33,0){\line(1,0){26}}
    \put(0,0){\circle{14}}
    \put(-4.5,-2.5){$s_1$}
    \put(7,0){\line(1,0){26}}
    \put(40,0){\circle{14}}
    \put(35.5,-2.5){$s_3$}
    \put(5,5){\line(1,1){10}}
    \put(20,20){\circle{14}}
    \put(15.5,17){$s_2$}
    \put(35,5){\line(-1,1){10}}
  \end{picture}
\end{figure}
\noindent
Conjugation by the first letter of a Coxeter
word will have the effect of moving this letter to the end of the word. For 
example, if $w = s_0 s_1 s_2 s_3$ then $s_0 w s_0 = s_1 s_2 s_3 s_0$. We call
this a {\em rotation} of the word. Say that two words are {\em rotation equivalent} if
one can be obtained from the other by a series of
rotations and commutations. For example,
$$ s_0 s_1 s_2 s_3 \sim s_1 s_2 s_3 s_0 \sim s_1 s_2 s_0 s_3 \sim s_2 s_0 s_3 s_1
$$
\noindent
Our aim is to prove the following characterization of conjugacy of
Coxeter elements.
\begin{theo}\label{main}
Coxeter elements are conjugate if and only if they are rotation equivalent.
\end{theo}
\noindent
We stated this result at the FPSAC meeting in 1994, but gave proofs only
for the two important special cases when the Coxeter graph is a tree or a 
cycle (covering all finite and affine groups). For these special cases,
the result has since been rediscovered by Shi  \cite{Shi}, who
extended it to cycles with trees attached. Here we present the first proof of
the general result.

\section{Edge orientations and chip-firing}
\noindent
For a graph $G$, an {\em acyclic edge orientation}
is an assignment of directions to all edges, such that the resulting digraph
is acyclic.  This is always possible.  A simple observation is that the
resulting digraph contains at least one {\em sink}, i.e.~a vertex with no
outgoing edges.

If each arrowhead is detached and pronounced a chip, we get a distribution
of chips on the vertices and can play the {\em chip-firing game}
introduced in \cite{BLS}.
Translated into edge orientations, a legal move consists in
choosing a sink and firing it, that is changing it into a source by reversing 
all its edges.  
Since neither sinks nor 
sources belong to any cycles, the graph will still be acyclic and contain
a sink, so the game goes on forever. 

Several authors have rediscovered and analysed this
edge reorientation game. 
All the following facts are consequences of Th.1 in \cite{Pretzel}.
\begin{propo}\label{playback}
  If a vertex $s$ is fired in an acyclic edge orientation, 
  there is a continuation in which {\em every other 
  vertex is fired exactly once}. Such a game sequence restores the
  original edge orientation.  
\end{propo}
\begin{proof}
  Induction over the number of vertices proves the proposition:
  After firing  $s$, use the induction hypothesis
  to fire all remaining nodes.  The base case is trivial as is the
  restoration of original orientation. 
\end{proof}
\begin{coro}
 There is a play sequence from $u$ to $v$ if and only if there is a play
 sequence from $v$ to $u$.
\end{coro}
\begin{proof}
If a single move can be inverted, so can a sequence of moves.  Thus, it is
  sufficient to consider the case when $v$ is the result of firing a single
  vertex in position $u$, so the proposition applies.
\end{proof}

\noindent
According to this result, reachability of positions in this game constitutes 
an equivalence
relation that partitions acyclic edge orientations into {\em reachability
classes}.
For many graphs, it is now a rather simple matter to enumerate acyclic edge
orientations and reachability classes.  Two basic cases are covered by
our next proposition.
\begin{propo}\label{reachability} 
  For a tree with $n$ nodes, there are $2^{n-1}$ acyclic edge orientations
  but only one reachability class.  For an $n$-cycle, there are $2^n-2$
  acyclic edge orientations and $n-1$ reachability classes of sizes
  $\binom{n}{1},\ldots, \binom{n}{n-1}$.
\end{propo}
\begin{proof}
  An $n$-vertex tree has got $n-1$ edges with no restrictions on orientations,
  for all directed trees are acyclic.
  The statement that all edge orientations are reachable from each other has 
  a simple induction proof: choose a leaf, play to give the rest of the
  tree the desired orientation, firing the leaf when necessary, finally
  fire the chosen leaf once more if needed to reorient its edge. 

  For an $n$-cycle, exactly two orientations are forbidden, namely all $n$
  clockwise or all $n$ anti-clockwise.  Consider the $\binom{n}{k}$ orientations 
  with $k$ anti-clockwise edges.  Firing a node may be seen as moving the 
  anti-clockwise arrow one step forward, e.g. 
  \begin{picture}(68,5)(-1,0)
    \put(13,3){\vector(-1,0){11}}
    \put(17,3){\vector( 1,0){11}}
    \put(43,3){\vector(-1,0){11}}
    \put(58,3){\vector(-1,0){11}}
    \put( 0,3){\circle*{2}}
    \put(15,3){\circle*{2}}
    \put(30,3){\circle*{2}}
    \put(45,3){\circle*{2}}
    \put(60,3){\circle*{2}}
  \end{picture}
  to
  \begin{picture}(64,5)(-1,0)
    \put(13,3){\vector(-1,0){11}}
    \put(28,3){\vector(-1,0){11}}
    \put(32,3){\vector( 1,0){11}}
    \put(58,3){\vector(-1,0){11}}
    \put( 0,3){\circle*{2}}
    \put(15,3){\circle*{2}}
    \put(30,3){\circle*{2}}
    \put(45,3){\circle*{2}}
    \put(60,3){\circle*{2}}
  \end{picture}

It is obvious that any position with $k$ anti-clockwise arrows can be reached
in this way. 
\end{proof}

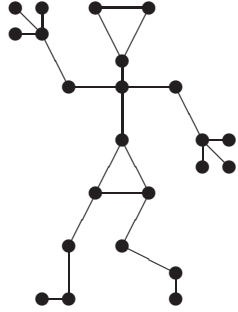
\begin{figure}[htb]
 \begin{picture}(80,110)(-40,0)
   
    \put(-30,0){\circle*{5}}
    \put(-20,0){\circle*{5}}
    \put(20,0){\circle*{5}}
    \put(-20,20){\circle*{5}}
    \put(0,20){\circle*{5}}
    \put(20,10){\circle*{5}}
    \put(-10,40){\circle*{5}}
    \put(10,40){\circle*{5}}
    \put(0,60){\circle*{5}}
    \put(0,80){\circle*{5}}
    \put(0,90){\circle*{5}}
    \put(-20,80){\circle*{5}}
    \put(20,80){\circle*{5}}
    \put(-10,110){\circle*{5}}
    \put(10,110){\circle*{5}}
    \put(-30,100){\circle*{5}}
    \put(-40,100){\circle*{5}}
    \put(-30,110){\circle*{5}}
    \put(-40,110){\circle*{5}}
    \put(30,60){\circle*{5}}
    \put(40,60){\circle*{5}}
    \put(30,50){\circle*{5}}
    \put(40,50){\circle*{5}}
    \put(-30,0){\line(1,0){10}}
    \put(-20,0){\line(0,1){20}}
    \put(-20,20){\line(1,2){10}}
    \put(-10,40){\line(1,0){20}}
    \put(20,0){\line(0,1){10}}
    \put(20,10){\line(-2,1){20}}
    \put(0,20){\line(1,2){10}}
    \put(-10,40){\line(1,2){10}}
    \put(10,40){\line(-1,2){10}}
    \put(0,60){\line(0,1){30}}
    \put(0,90){\line(-1,2){10}}
    \put(0,90){\line(1,2){10}}
    \put(-10,110){\line(1,0){20}}
    \put(-20,80){\line(1,0){40}}
    \put(-20,80){\line(-1,2){10}}
    \put(20,80){\line(1,-2){10}}
    \put(-30,100){\line(-1,0){10}}
    \put(-30,100){\line(-1,1){10}}
    \put(-30,100){\line(0,1){10}}
    \put(30,60){\line(1,0){10}}
    \put(30,60){\line(1,-1){10}}
    \put(30,60){\line(0,-1){10}}
  \end{picture}
\caption{A trunk with four limbs and three joints. }
\end{figure}

\noindent 
A connected graph that is not a tree may be decomposed uniquely as a 
leafless {\em trunk} (the subgraph obtained by successive removal of
leaves until none are left) and a collection of
{\em limbs}, defined as trees that connect to the trunk at one
vertex only, called a {\em joint}.
We may then apply the same induction argument as we used for trees 
to obtain the following useful result. 
\begin{propo}\label{limbs}
In an acyclic edge orientation, if edges on limbs are arbitrarily
redirected, the result is a new acyclic orientation in the same reachability
class.
\end{propo}

\section{Words with intervening neighbours}
\noindent
Let $G$ be the Coxeter graph of a Coxeter group with generators 
$S$. Consider a word $w$ in the alphabet $S$. If there is an edge between
$s$ and $t$ and if the first occurrence of $s$
in $w$ precedes the first occurrence of $t$, we orient the edge like
$s\!\rightarrow\! t$. In this way we edge-orient the subgraph of $G$ spanned 
by the letters in $w$.

Now, consider this edge orientation as a right-to-left process on the
word $w$. The rightmost letter orients no edge, the two rightmost letters
orient the edge between the corresponding vertices (if there is one) and the 
larger the segment, the more edges get 
oriented. When a letter $t$ reappears, we may have to reverse some arrows 
$s\!\rightarrow \! t$, namely when the new situation is 
$t\cdots s \cdots t \cdots$, and if all $t$-neighbours occur in between the 
first and the second $t$, this will be a chip-firing move. 
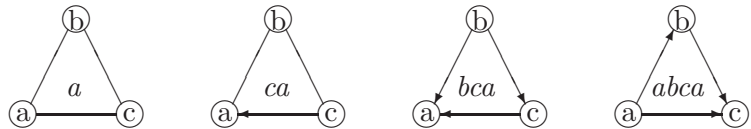
\begin{figure}[htb]

 \begin{picture}(50,35)(-9,0)
    \put(0,0){\circle{10}}
    \put(-2.6,-2.6){a}
    \put(5,0){\line(1,0){30}}
    \put(40,0){\circle{10}}
    \put(37.4,-2.6){c}
    \put(3,4){\line(1,2){14}}
    \put(20,36){\circle{10}}
    \put(17,31.8){b}
    \put(37,4){\line(-1,2){14}}
    \put(16,7){\it a}
  \end{picture}
\qquad
 \begin{picture}(50,35)(-9,0)
    \put(0,0){\circle{10}}
    \put(-2.6,-2.6){a}
    \put(35,0){\vector(-1,0){30}}
    \put(40,0){\circle{10}}
    \put(37.4,-2.6){c}
    \put(3,4){\line(1,2){14}}
    \put(20,36){\circle{10}}
    \put(17,32){b}
    \put(37,4){\line(-1,2){14}}
    \put(14,7){\it ca}
  \end{picture}
\qquad
 \begin{picture}(50,35)(-9,0)
    \put(0,0){\circle{10}}
    \put(-2.6,-2.6){a}
    \put(35,0){\vector(-1,0){30}}
    \put(40,0){\circle{10}}
    \put(37.4,-2.6){c}
    \put(17,32){\vector(-1,-2){14}}
    \put(20,36){\circle{10}}
    \put(17,32){b}
    \put(23,32){\vector(1,-2){14}}
    \put(11,7){\it bca}
  \end{picture}
\qquad
 \begin{picture}(50,35)(-9,0)
   
    \put(0,0){\circle{10}}
    \put(-2.6,-2.6){a}
    \put(5,0){\vector(1,0){30}}
    \put(40,0){\circle{10}}
    \put(37.4,-2.6){c}
    \put(3,4){\vector(1,2){14}}
    \put(20,36){\circle{10}}
    \put(17,32){b}
    \put(23,32){\vector(1,-2){14}}
    \put(8,7){\it abca}
  \end{picture}
\caption{Successive edge orientation by the word $abca$ }
\end{figure}

\begin{definition}
A word has the 
{\em intervening neighbours} property if any two occurrences of
the same letter are separated by all its graph neighbours.
\end{definition}

\noindent
If $w$ has this property and if all letters of $S$ occur in $w$, 
eventually the right-to-left process will have oriented all edges
in $G$, so giving such a word $w$ is equivalent to giving an initial
edge orientation and a play sequence.

If two words, $w$ and $w^\prime$, represent the
the same group element, are the corresponding edge orientations 
necessarily in the same reachability class? For a leafless graph the 
answer is yes. 

\begin{propo}\label{intervening}
Let $G$ be a trunk without limbs and let $w$ be a word with the intervening 
neighbours property in which all letters of $S$ occur. Then $w$ is a
reduced word for the group element it represents, all reduced
words for this element are obtained by commutations in $w$ and all edge 
orientations defined by these words belong to the same reachability class.
\end{propo}
\begin{proof}
When two occurrences of the same letter are separated by two or more
neighbours, no braid relations such as $sts=tst$ apply, so commutations 
are the only applicable rewriting rules. Commutations preserve the 
intervening neighbours property and no reduction is possible. Nor do 
commutations affect the edge orientation.
\end{proof}
Now let $G$ be a general graph, regarded as a trunk with limbs. 
It is no longer true that the
intervening neighbours property is an invariant under rewritings ---
for example, if $s$ is a leaf connected to the trunk vertex $t$, 
the braid transformation $sts=tst$ will produce two occurrences of 
$t$ with only one intervening neighbour. It turns out that only limb 
letters are involved in braid transformations and that the intervening 
neighbours property stays true for the other trunk letters, with a 
slight modification for the joints
(vertices in which a limb connects with the trunk). Obviously, 
trunk letters never occur in higher braid transformations like $stst=tsts$. 

The following lemma states properties that are true for a 
word with the intervening neighbours property and which stay true
under rewritings.
\begin{lemma}\label{invariants}
The following word properties are invariant under commutations $st=ts$ and
braid transformations $sts=tst$.
\begin{itemize}
\item The intervening neighbours property holds for trunk letters that
are not joints, i.e.~any two occurrences of such a letter are separated by
all its neighbours.
\item Two occurrences of the same joint are either separated by all its 
trunk neighbours (it has at least two) or by no trunk neighbour (but by
at least one limb neighbour). 
\end{itemize}
\end{lemma}
\begin{proof}
Invariance under commutations is trivial. A braid transformation must
involve two limb letters (one of which may be a joint) so the first
property stays true. If $s$ is a joint or $t$ is a joint, the second
property still stays true after $sts=tst$. 
\end{proof}
\noindent
We are now almost ready to extend Prop.~\ref{intervening} to trunks with
limbs. The missing piece was provided by David Speyer \cite{Speyer}.
\begin{lemma}
[Speyer, 2008] For infinite irreducible Coxeter groups, all words with the 
intervening neighbours property are reduced.
\end{lemma}
\noindent
Speyer actually states the result for {c-admissible sequences}, that is 
valid play sequences from edge orientation $c$, but as we have noted,
the concepts are equivalent.

\begin{propo}\label{intervening2}
Let $G$ be any connected non-tree graph and let $w$ be a word 
with the intervening neighbours property in which all letters of $S$ occur. 
Then $w$ is a reduced word, all reduced words for this element are obtained 
by commutations in $w$ and braid transformations involving only limb letters,
and all edge 
orientations defined by these words belong to the same reachability class.
\end{propo}
\begin{proof}
All non-tree Coxeter graphs define infinite groups, so Speyer's result 
applies.
Observe that, by Prop.~\ref{limbs}, the orientation of limb edges
is insignificant for reachability, and neither commutations nor braid 
transformations of limb letters influence the edge orientations in the
trunk. Hence, we may disregard all limb letters except for the joints. 
According to Lemma~\ref{invariants}, the intervening neighbours property with 
respect to the set of trunk letters that are not joints will hold under
rewriting. Joints may duplicate, but as there are only limb neighbours
between the duplicates, the argument in the proof of Prop.~\ref{intervening}
goes through with respect to the trunk letters.
When two occurrences of the same letter are separated by two or more
neighbours, no braid relations of type $sts=tst$ apply, so commutations 
are the only applicable rewriting rules. Commutations preserve the 
intervening neighbours property and no reduction is possible. Nor do 
commutations affect the edge orientation.
\end{proof}

\section{Coxeter elements}
\noindent
A Coxeter word has one instance of each letter, so it defines an
acyclic orientation of the Coxeter graph.
This orientation is in fact well defined by the Coxeter {\em element}, as 
all words representing the same element are obtainable by commutations.   
And in Prop.~\ref{playback} we noted that for every acyclic
orientation there is a play sequence in which all vertices are fired once,
i.e.~a corresponding Coxeter word.

So, Coxeter elements correspond bijectively to acyclic edge orientations 
of the Coxeter graph. And we have seen that moving the first letter of a 
Coxeter word to the end is the same as firing the corresponding vertex.
This proves the following proposition from \cite{EriEri}.
\begin{propo}\label{rotation}
  Rotation of Coxeter words induces an equivalence relation on the set of
  Coxeter elements, that corresponds precisely to the reachability relation
  on the set of acyclic edge orientations.
\end{propo}

\noindent
{\bf Proof of theorem \ref{main}.}
Our main theorem states that reachability classes and conjugacy classes
coincide. Rotating a letter $s$ from the beginning to the end is the same 
thing as conjugating by $s$, therefore rotation equivalent elements are 
indeed conjugate. Proving that two conjugate Coxeter elements
$w$ and $w^\prime=uwu^{-1}$ must belong to the same reachability class 
is harder, but because of Prop.~\ref{reachability} we need only
prove it for connected non-tree graphs.

The trick is to consider a power $(w^\prime)^k = uw^ku^{-1}$ with $k$
sufficiently large. Note that $(w^\prime)^k$ and $uw^ku^{-1}$ are two
different words for the same group element, so the second one must be
reducible. The proof will have three steps:
\begin{enumerate}
\item The word $(w^\prime)^k$ is reduced. 
\item The word $uw^ku^{-1}$ has a reduced form
$u_1 w u_2$.
\item Since $w$ and $w^\prime$ appear in words representing the
same element, they belong to the same reachability class and are
rotation equivalent (as explained in detail below).
\end{enumerate}

\bigskip
\noindent
We regard the graph as a trunk with limbs. The proof is easiest if there are 
no limbs, but in the end, the limbs will turn out to be of no
consequence. 
\begin{enumerate}
\item Since $(w^\prime)^k$ has the intervening neighbours property, 
Prop.~\ref{intervening2} implies that $(w^\prime)^k$ is reduced,

\item The well-known {\em deletion property} for Coxeter groups (see 
\cite{BB})
states that any word can be brought to a reduced form through a series of 
successive deletions of pairs of letters (not necessarily adjacent).
For $uw^ku^{-1}$, the number of such deletions is the same as the number
of letters in $u$. For any $k$, greater than this number, at least one
instance of $w$ will remain intact after the deletions.

\item Prop.~\ref{intervening2} tells us that all
edge orientations obtained from the words $(w^\prime)^k$ and
$u_1 w u_2$  by the right-to-left process are in the same
reachability class. When the sequence $w$ in the middle of the second
word has just been processed, the edge orientation is of course 
completely defined by $w$. Therefore  $w^\prime$ and $w$ are rotation 
equivalent. \hfil\qed
\end{enumerate}

\noindent
{\bf Note.} Pretzel \cite{Pretzel} showed that two acyclic orientations
belong to the same class if they have the same circulation around every
cycle, so we have a very explicit characterization of the conjugacy
classes.

\noindent
{\bf Acknowledgments.} We thank Henning Mortveit and Matthew Macauley
for urging us to publish the lost proof.


\begin{thebibliography}{ABC}

\bibitem{BB} A. Bj{\"o}rner and F. Brenti, {\sl Combinatorics of Coxeter 
Groups}, Grad.~Texts in Math., vol.~231, Springer, 2005.

\bibitem{BLS} A. Bj{\"o}rner, L. Lov\'asz and P. W. Shor,
{\sl Chip-firing games on graphs}, European J. Combin. {\bf 12}, 
1991, 283--291.

\bibitem{EriEri}
H. Eriksson and K. Eriksson.
\newblock Chip-firing and Coxeter elements.
\newblock In {\em Proceedings of the 6th
    conference on Formal Power Series and Algebraic Combinatorics},
    145--151, DIMACS, 1994.

\bibitem{Shi} Jian Yi Shi, {\sl Conjugacy relation on Coxeter elements}, Advances in Mathematics {\bf 161}, 2001, 1--19.
\bibitem{Pretzel} O. Pretzel, {\sl On reorienting graphs by pushing down maximal vertices}, Order {\bf 3}, 1986, no.~2, 135--153.

\bibitem{Speyer} D. Speyer, {\sl Powers of Coxeter elements in infinite
groups are reduced}, Proc. Amer. Math. Soc. {\bf 137}, 2009, 1295--1302.

\end{thebibliography}
\end{document}